\documentclass[letterpaper, 10 pt, conference]{ieeeconf}  

\IEEEoverridecommandlockouts                              
\usepackage{hyperref}
\hypersetup{hidelinks}

\usepackage{cite}

\usepackage{amsthm}
\usepackage{amsmath} 
\usepackage{amssymb}

\let\labelindent\relax 
\usepackage{enumitem}
\setlist[enumerate]{
  labelsep=8pt,
  labelindent=2\parindent,
  itemindent=0pt,
  leftmargin=*,
  before=\setlength{\listparindent}{-\leftmargin},
}

\usepackage{tikz}

\usepackage[all]{xy} 

\usetikzlibrary{arrows}
\usetikzlibrary{graphs,decorations.pathmorphing,decorations.markings}
\usetikzlibrary{calc}

\pgfmathsetmacro\weight{1/2}
\pgfmathsetmacro\third{1/3}
\pgfmathsetmacro\twothirds{2/3}

\tikzset{degil/.style={
            decoration={markings,
            mark= at position 0.5 with {
                  \node[transform shape] (tempnode) {$/$};
                  }
              },
              postaction={decorate}
}
}

\newtheorem{theorem}{Theorem}[section]
\newtheorem{lemma}[theorem]{Lemma}

\newtheorem{definition}[theorem]{Definition}

\newtheorem{remark}[theorem]{Remark}

\newcounter{syscounter}

\newcommand \N   {\mathbb{N}}
\newcommand \R   {\mathbb{R}}

\newcommand \K   {\mathcal{K}}
\newcommand \Kinf{\mathcal{K_\infty}}

\newcommand \KL  {\mathcal{KL}}
\newcommand \LL  {\mathcal{L}}

\newcommand{\thetamax}{\theta_{p}}

\newcommand \eps {\varepsilon}

\newcommand{\XX}{\mathcal{X}}
\newcommand{\UU}{\mathcal{U}}
\newcommand{\VV}{\mathcal{V}}

\renewcommand{\leq}{\leqslant}
\renewcommand{\geq}{\geqslant}

\newcommand{\prop}[4]{(\textnormal{#1}, #2, #3)\text{ is }\textnormal{#4}}

\newcommand{\xinffc}{\prop{\ref{syst}}{\XX^\infty}{\UU}{FC}}
\newcommand{\xinfbrs}{\prop{\ref{syst}}{\XX^\infty}{\UU}{BRS}}
\newcommand{\xzerofc}{\prop{\ref{syst}}{\XX^0}{\UU}{FC}}
\newcommand{\xzerobrs}{\prop{\ref{syst}}{\XX^0}{\UU}{BRS}}
\newcommand{\zinffc}{\prop{\ref{syst_input}}{\R^n}{\VV^\infty\times\UU}{FC}}
\newcommand{\zinfbrs}{\prop{\ref{syst_input}}{\R^n}{\VV^\infty\times\UU}{BRS}}
\newcommand{\zzerofc}{\prop{\ref{syst_input}}{\R^n}{\VV^0\times\UU}{FC}}
\newcommand{\zzerobrs}{\prop{\ref{syst_input}}{\R^n}{\VV^0\times\UU}{BRS}}

\newcommand{\propnou}[3]{(\textnormal{#1}, #2)\text{ is }\textnormal{#3}}

\newcommand{\xinffcnou}{\propnou{\ref{syst}}{\XX^\infty}{FC}}
\newcommand{\xinfbrsnou}{\propnou{\ref{syst}}{\XX^\infty}{BRS}}

\newcommand{\xzerobrsnou}{\propnou{\ref{syst}}{\XX^0}{BRS}}

\newcommand{\gas}{\propnou{\ref{syst_noinputs}}{\XX^\infty}{GAS}}
\newcommand{\ugas}{\propnou{\ref{syst_noinputs}}{\XX^\infty}{UGAS}}
\newcommand{\zerogas}{\propnou{\ref{syst_noinputs}}{\XX^0}{GAS}}
\newcommand{\zerougas}{\propnou{\ref{syst_noinputs}}{\XX^0}{UGAS}}

\newif\ifAndo

\Andotrue

\title{\LARGE \bf
Forward completeness implies bounded reachable sets
for
time-delay systems
on the state space of 
essentially bounded measurable functions
}

\author{
Lucas Brivadis,
Antoine Chaillet,
Andrii Mironchenko and Fabian Wirth\\
\thanks{L. Brivadis and A. Chaillet are with Université Paris-Saclay, CNRS, CentraleSupélec, Laboratoire des Signaux et Systèmes, 91190, Gif-sur-Yvette, France. Emails: {\tt\small (firstname).(lastname)@centralesupelec.fr}}
\thanks{A. Mironchenko is with the Department of Mathematics, University of Klagenfurt, 9020, Klagenfurt, Austria. 
Email:
{\tt\small andrii.mironchenko@aau.at}
}
\thanks{F. Wirth is with Faculty of Computer Science and Mathematics, University of Passau,
Innstra\ss e 33, 94032 Passau, Germany. Email:
{\tt\small fabian.(lastname)@uni-passau.de}
}
\thanks{This research has been supported by BayFrance (Franco-Bavarian University cooperation center) within the project ``Input-to-State Stability of Systems with Delays'', project FK-20-2022. A. Mironchenko has been supported by the German Research Foundation (DFG) (grant MI 1886/2-2).}
}

\begin{document}

\maketitle
\thispagestyle{empty}
\pagestyle{empty}

\begin{abstract}

We consider time-delay systems with a finite number of delays in the state space $L^\infty\times\R^n$. In this framework, we show that forward completeness implies the bounded reachability sets property, while this implication 
was recently shown by J.L. Mancilla-Aguilar and H. Haimovich to fail in the state space of continuous functions.
As a consequence, we show that global asymptotic stability is always uniform in the state space $L^\infty\times\R^n$.

\end{abstract} 








\section{Introduction}




A control system is called forward complete (FC)
if for any initial condition $x_0$, and any input $u$, the corresponding trajectory $x(\cdot;x_0,u)$ is well-defined on the whole nonnegative time axis $\R_+$.
If additionally, for any magnitude $r>0$ and any time $T>0$,
\[
\sup_{\|x_0\| \leq r,\ \|u\| \leq r,\ t\in[0,T]} |x(t;x_0,u)| <+\infty,
\] 
then a control system is said to have bounded reachability sets (BRS). 


BRS establishing uniform bounds for solutions on finite time intervals is a bridge between the pure well-posedness theory (that studies existence and uniqueness but does not care much about the bounds for solutions) and the stability theory (which is interested in establishing certain bounds for solutions for all nonnegative times, as well as their convergence). BRS (and the closely related notion of robust forward completeness \cite{KaJ11b}) proved to be useful in many contexts, such as derivation of converse Lyapunov theorems for global asymptotic stability \cite{LSW96}, characterization of input-to-state stability for nonlinear systems \cite{MiW18b}, non-coercive Lyapunov methods \cite{MiW19a, MiW18b, JMP20}, or criteria for global asymptotic stability for retarded systems \cite{karafyllis2022global}, to name a few.

This motivated the development of criteria for the BRS property. In \cite{LSW96}, it was shown that FC and BRS are equivalent properties for ordinary differential equations (ODEs) with Lipschitz continuous right-hand side and Lebesgue measurable essentially bounded inputs.
If the inputs are a priori uniformly bounded in magnitude, and the dynamics are locally Lipschitz (jointly in the state and inputs), these properties are also equivalent to the existence of a Lyapunov function that increases at most exponentially along solutions \cite{AnS99}.
A Lyapunov criterion for robust forward completeness of general infinite-dimensional systems has been shown in \cite{Mir23e}.

Recent studies have revealed that for inifinite-dimensional systems, the relation between FC and BRS is rather complex. 
For linear infinite-dimensional systems, FC does imply BRS \cite[Proposition 2.5]{Wei89b}. However, even in the absence of inputs, this implication no longer holds when dealing with nonlinear infinite-dimensional systems, as demonstrated in \cite[Example 2]{MiW18b} for infinite ODE networks (aka ensembles).
In the specific context of time-delay systems, it was recently proved in \cite{MaH23} that FC does not necessarily ensure BRS when the considered state space is given by continuous functions. 
Nevertheless, when considering more regular state spaces, such as Sobolev and Hölder spaces, FC and BRS properties turn out to be equivalent \cite{karafyllis2022global}
for autonomous systems without inputs (the question remains open for systems with inputs and would require adapting the proof of \cite{LSW96} while using the compactness arguments of \cite{karafyllis2022global}).


In this paper, \emph{we consider the evolution of time-delay systems on a wider state space, namely $L^\infty\times\R^n$. In this setup, we show that FC is again equivalent to BRS.}
We introduce time-delay systems on this new state space in Section \ref{sec:linf} and relate solutions of time-delay systems with solutions of ODEs with well-chosen inputs.
Our main result is stated and proved in Section \ref{sec:main}.
In Section \ref{sec:gas}, we prove that global asymptotic stability is necessarily uniform in this particular state space.

\vspace{3mm}
\subsubsection*{Notation}

Let $I$ be an interval of $\R$ and $n$ be a positive integer.
For a given function $f:I\to \R^n$, we say that $f$ belongs to
$\LL^\infty(I, \R^n)$ if it is Lebesgue measurable and essentially bounded, and to $C^0(I, \R^n)$ if it is continuous.
The restriction of $f$ to $J\subset I$ is denoted by $f|_J$.
We denote by $L^\infty(I, \R^n)$ (resp. $L^\infty_\textrm{loc}(I, \R^n)$) the space of Lebesgue measurable essentially bounded (resp. essentially bounded on any bounded subset of $I$) functions quotiented by the space of almost always null functions.
Moreover, let $\K$ be the set of continuous functions from $\R_+$ to $\R_+$ that are increasing and null at $0$, 
and $\Kinf$ the set of unbounded functions in $\K$.
We say that a continuous function $\beta:\R_+^2\to \R_+$ is of class $\KL$ if $\beta(\cdot, s)$ is of class $\K$ for all $s\in\R_+$ and $\beta(r,\cdot)$ is decreasing and tends towards $0$ at infinity for all $r>0$.

\section{Time-delay systems in \texorpdfstring{$L^\infty\times\R^n$}{Linf x Rn}}
\label{sec:linf}

Let $n, m$ and $p$ be positive integers.
Consider a control system with a finite number of delays
\begin{equation}\label{syst}
\tag{TDS}
\dot x(t) = f(x(t), (x(t-\theta_k))_{k\in\{1,\dots,p\}}, u(t)),
\end{equation}
where $x(t)\in\R^n$,
$u(t)\in\R^m$ is the input, $(\theta_k)_{k\in\{1,\dots,p\}}\in\R_{>0}^p$ are the delays,
and $f:\R^n\times (\R^n)^{p}\times\R^m\to\R^n$ is continuous and locally Lipschitz with respect to its first variable (in $\R^{n}\times (\R^n)^{p}$), uniformly with respect to the last variables (in $\R^m$).
With no loss of generality, we may assume that all the delays are distinct
and that $0<\theta_1<\dots<\theta_p$.
We consider $\UU = L^\infty_\textrm{loc}((0, +\infty), \R^m)$ to be the input space, as commonly assumed in the robustness analysis literature \cite{CKP23}.

In this paper, we consider two different state spaces. 
The first one is the usual space of continuous functions $$\XX^0:=C^0([-\thetamax, 0], \R^n).$$
The second one aims at allowing a wider class of initial states, namely Lebesgue measurable essentially bounded signals. However, one cannot simply consider $L^\infty((-\thetamax, 0), \R^n)$ as a state space because the value of the initial condition at the initial time must be specified for the Cauchy problem to make sense. For this reason, and following \cite{DELFOUR1972213, bensoussan2007representation}, we define $\XX^\infty$ as the space of measurable essentially bounded functions
quotiented by the space of almost everywhere null functions that are moreover null at 0:
\begin{align*}
\XX^\infty:=
&
\LL^\infty([-\thetamax, 0], \R^n)
\\
&\qquad
\Big/
\Big\{
x_0\in\LL^\infty([-\thetamax, 0], \R^n)\Big\vert
\begin{cases}
\|x_0\|_{L^\infty} =0,
\\
x_0(0)=0
\end{cases} \hspace{-2mm}
\Big\}.
\end{align*}
One can easily check that $\XX^\infty$ endowed with the norm defined by
$\|x_0\|_{\XX^\infty} := \max(\|x_0\|_{L^\infty}, |x_0(0)|)$
for all $x_0\in\XX^\infty$
is a Banach space, and that it is isometrically isomorphic to $L^\infty((-\thetamax, 0), \R^n)\times\R^n$.
Roughly speaking, an element in $\XX^\infty$ is an element of $x\in L^\infty((-\thetamax, 0), \R^n)$ as well as a point $x(0)\in\R^n$ defining the value of the function $x$ at $0$.

Let us introduce the following notion of solution.


\begin{definition}[$\XX$-solutions]\label{def:sol}
Let $\XX$ be either $\XX^0$ or $\XX^\infty$.
Consider an initial condition $x_0\in\XX$.
Let $u\in \UU$ and $T\in\R_{>0}\cup\{+\infty\}$.
We say that $x\in C^0([0, T), \R^n)$ is a \emph{$\XX$-solution over $[0, T)$} to the Cauchy problem \eqref{syst} initialized at $x_0$ if $x(0) = x_0(0)$, $x$ is absolutely continuous on any compact subinterval of $[0, T)$,
and satisfies 
\begin{equation*}
\dot x(t) = f(x(t), ((x_0\diamond x)(t-\theta_k))_{k\in\{1,\dots,p\}}, u(t))
\end{equation*}
for almost all $t\in (0, T)$. Here $x_0\diamond x:[-\thetamax, T)\to\R^n$ is defined almost everywhere by
\begin{equation*}
    (x_0\diamond x)(s) := \begin{cases}
        x(s),&\text{for all } s\in(0, T),
        \\
        x_0(s),&\text{for almost all } s\in[-\thetamax, 0].
    \end{cases}
\end{equation*}
\end{definition}




\begin{remark}\label{rem:cont}
Note that $\XX^0$-solutions are the usual continuous solutions of \eqref{syst} (see \cite[Part 2, Chapter 4.3]{bensoussan2007representation}).
Moreover, both for $\XX^0$ and $\XX^\infty$-solutions, we have that
solutions are continuous functions of time from $t=0$.
\end{remark}

Although the results guaranteeing the existence and uniqueness of solutions are known for initial conditions in $\mathcal X^0$ (see, e.g., \cite[Chapter 2.2]{Hal77} or \cite[Part II, Chapter 4.3.1]{bensoussan2007representation}),
the case of $\XX^\infty$ received less attention. Nevertheless, we show below that existing results can easily be adapted to this state space in the case of a finite number of delays.



\begin{theorem}[Existence, uniqueness]\label{thm-existence}
Let $\XX$ be either $\XX^0$ or $\XX^\infty$.
Given any initial conditions $x_0\in\XX$ and any $u\in\UU$, 
there exists the unique maximal $\XX$-solution $x(\cdot,x_0,u)$ of
the Cauchy problem \eqref{syst} initialized at $x_0$ corresponding to an input $u$. 
Denote $x_t(x_0, u): [-\thetamax, 0]\ni \theta \mapsto x(t+\theta; x_0, u)$.

Furthermore, the system \eqref{syst}
 is a control system in the sense of \cite{MiP20}, except that solutions may not be continuous functions in the topology of $\XX$. We denote this system $(\textnormal{\ref{syst}},\XX,\UU)$.
\end{theorem}

In other words, $x_t(x_0, u)$ is the flow of the control system, taking values in a functional state space, while $x(t; x_0, u)$ gives the value of the solution at each time $t$, i.e., is the evaluation of $x_t(x_0, u)(0)$. 







\begin{proof}
For $\XX=\XX^0$, see, e.g., \cite[Part II, Chapter 4.3.1]{bensoussan2007representation}.
For $\XX=\XX^\infty$, the result follows \cite{DELFOUR1972213} by remarking that \cite[Proposition 4.3]{DELFOUR1972213}
holds for $p=\infty$ for systems in the form of \eqref{syst} (i.e. with a finite number of delays). Alternatively, it follows from Lemma \ref{lem:x_z} (proved later in the paper),
with $T=\min(\theta_1, \min_{\substack{k, j\in\{1, \dots, p\}\\ k\neq j}}|\theta_k-\theta_j|)$.
\end{proof}
\begin{remark}\label{rem:flow}
For initial conditions $x_0$ in $\XX^\infty$, note that
the flow map $t\mapsto x_t(x_0, u)$ is not necessarily continuous (continuity of the flow is not required in Definition \ref{def:sol}). For example, since the shift operator is not continuous on $L^\infty(\R, \R^n)$, the flow
associated
to $\dot x = 0$ is not continuous at the initial condition 
\[
x_0:s\mapsto\begin{cases}
    0 &\text{if } -\theta_p\leq s\leq -\theta_p/2
    \\
    1 &\text{if } -\theta_p/2< s\leq 0
\end{cases}.
\]
However, since solutions are continuous from $t=0$ (see Remark \ref{rem:cont}), the flow is continuous over $[\theta_p, +\infty)$.
\end{remark}

We now extend several properties of time-delay systems that are usually defined only for $\XX^0$-solutions \cite{CKP23}. We start with forward completeness.





\begin{definition}
    Let $\XX$ be either $\XX^0$ or $\XX^\infty$.
    We say that $(\textnormal{\ref{syst}},\XX,\UU)$ is \emph{forward complete (FC)} if for all $x_0\in\XX$ and all $u\in \UU$, the corresponding solution $x(\cdot,x_0,u)$ of $(\textnormal{\ref{syst}},\XX,\UU)$ exists on $[0, +\infty)$.
\end{definition}



We may additionally request that, over any bounded time interval, the solutions generated from a bounded set of initial states, with inputs taking values in any given bounded set, cover only a bounded subset of the state space. This corresponds to the BRS property, also sometimes referred to as robust forward completeness in the literature \cite{karafyllis2022global, CKP23, MaH23}.


\begin{definition}
    Let $\XX$ be either $\XX^0$ or $\XX^\infty$.
    We say that $(\mathrm{\ref{syst}},\XX,\UU)$ has \emph{bounded reachability sets (is BRS)}
    if
    it is \textnormal{FC} and for any $r>0$, the set
    \begin{align*}
    \{x(t; x_0, u)\mid
    &\ t\in[0, r],\quad
    x_0\in\XX, \quad u\in\UU,
    \\
    &\qquad \|x_0\|_\XX\leq r,\quad  \|u\|_{L^\infty}\leq r \}
    \end{align*}
    is bounded.
\end{definition}

\begin{remark}
    By causality, one can equivalently write $\|u|_{[0, r]}\|_{L^\infty}\leq r$ instead of $\|u\|_{L^\infty}\leq r$ in the above condition without changing the definition of \textnormal{BRS}.
\end{remark}

An equivalent characterization of BRS is given by the following lemma, which requires only FC and bounded reachable sets on a (possibly short) time interval.
\begin{lemma}\label{lem:relax_brs}
$\prop{\ref{syst}}{\XX}{\UU}{BRS}$ if and only if it is \textnormal{FC} and there exists $T>0$ such that for all $r>0$ the following set is bounded:
\begin{align*}
    \{x(t; x_0, u)\mid
    &\ t\in[0, T], \quad 
    x_0\in\XX, \quad u\in\UU,
    \\
    &\qquad \|x_0\|_\XX\leq r, \quad  \|u\|_{L^\infty}\leq r \}.
    \end{align*}
\end{lemma}
\begin{proof}
We proceed by induction. Assume that for some $n\in\N_{>0}$ and all $r>0$, the set
\begin{align*}
    S_{n,r}=\{x(t; x_0, u)\mid
    &\ t\in[0, nT], 
    x_0\in\XX, u\in\UU,
    \\
    &\|x_0\|_\XX\leq r,  \|u\|_{L^\infty}\leq r \}
    \end{align*}
    is bounded. Define
    $\rho_{n, r} = \max(r, \sup_{\xi\in S_{n,r}} |\xi|)$.

By the cocycle property of the flow,
$x(t+T; x_0, u) = x(T; x_t(x_0, u), u(t+\cdot))$ for all $t,T>0$, all $x_0\in\XX$ and all $u\in\UU$.
Hence, for all $r>0$,
\begin{align*}
&S_{n+1,r}
\\
&\begin{aligned}
    \subset S_{n,r} \cup
    \{x(t; x_0, u)\mid
    &\, t\in[nT, (n+1)T], 
    x_0\in\XX, u\in\UU,
    \\
    &\|x_0\|_\XX\leq r,  \|u\|_{L^\infty}\leq r \}
\end{aligned}
\\
&\begin{aligned}
    \subset S_{n,r} \cup
    \{&x(t; x_{nT}(x_0, u), u(nT+\cdot))\mid
    \\
    &\ t\in[0, T], 
    x_0\in\XX, u\in\UU,
    \\
    &\ \|x_0\|_\XX\leq r,  \|u\|_{L^\infty}\leq r \}
\end{aligned}
\\
&\begin{aligned}
    \subset S_{n,r} \cup
    \{&x(t; x_0, u)\mid
    \\
    &\ t\in[0, T], 
    x_0\in\XX, u\in\UU,
    \\
    &\ \|x_0\|_\XX\leq \rho_{n, r},  \|u\|_{L^\infty}\leq r \}
\end{aligned}
\\
&\begin{aligned}
    \subset S_{n,r} \cup
    S_{1, \rho_{n, r}}.
\end{aligned}
\end{align*}
Hence, if $S_{1,r}$ is bounded for all $r>0$, so is $S_{n,r}$ for all $n\in\N_{>0}$ and all $r>0$, thus the system is BRS.
\end{proof}

To investigate the relation between FC and BRS for the state spaces $\XX^0$ and $\XX^\infty$, we consider the following ODE with inputs associated to \eqref{syst}:
\begin{equation}\label{syst_input}
\tag{ODE}
\dot z(t) = f(z(t), (v_k(t))_{k\in\{1,\dots,p\}}, u(t)),
\end{equation}
where $v = (v_k)_{k\in\{1,\dots,p\}}$ and $u$ are inputs.
Define $$\VV^0 = C^0(\R_+, (\R^n)^p)\quad \text{and} \quad \VV^\infty=L^\infty_\textrm{loc}(\R_+, (\R^n)^p),$$
and assume that $v\in\VV$ where $\VV$ is either $\VV^0$ or $\VV^\infty$.
Recall that $f$ is continuous and locally Lipschitz continuous with respect to its first variables (in $\R^{n}$), uniformly with respect to the last variables (in $ (\R^n)^{p}\times \R^m$). Hence, according to the Cauchy-Lipschitz theorem,
the Cauchy problem associated to system \eqref{syst_input} with an initial condition $z_0\in\R^n$ and an input $(v, u)\in \VV\times \UU$ admits a unique maximal solution $z(\cdot; z_0, v, u)\in C^0([0, T), \R^n)$ for some $T\in\R_{>0}\cup\{+\infty\}$.
We denote this control system $(\textnormal{\ref{syst_input}}, \R^n, \VV\times\UU)$.
We recall the definition of FC and BRS for this system, and distinguish between two different classes of inputs.



\begin{definition}
Let $\VV$ be either $\VV^0$ or $\VV^\infty$.
We say that $(\textnormal{\ref{syst_input}}, \R^n, \VV\times\UU)$
is \emph{forward complete (FC)} if for all  $z_0\in\R^n$, all $v\in\VV$, and all $u\in \UU$, the corresponding maximal solution of $(\textnormal{\ref{syst_input}}, \R^n, \VV\times\UU)$ exists on $[0, +\infty)$.
\end{definition}

\begin{definition}
    Let $\VV$ be either $\VV^0$ or $\VV^\infty$.
    We say that $(\textnormal{\ref{syst_input}}, \R^n, \VV\times\UU)$ has the \emph{bounded reachability sets property (is BRS)} if
    it is \textnormal{FC} and for any $r>0$, the set
    \begin{align*}
    \{z(t; z_0, v, u)\mid
    &\ t\in[0, r], 
    z_0\in\R^n, v\in\VV, u\in\UU,
    \\
    &|z_0|\leq r, 
    \|v\|_{L^\infty}\leq r,
    \|u\|_{L^\infty}\leq r \}
    \end{align*}
    is bounded.
\end{definition}


Our motivation for introducing system \eqref{syst_input} is the following lemma that relates solutions of \eqref{syst} with those of \eqref{syst_input}. Our investigation of FC and BRS of \eqref{syst} relies on the study of FC and BRS of \eqref{syst_input}, as was already done in \cite[Theorem 9]{MaH23} for the state space $\mathcal X^0$ and input space $\VV^\infty$.

\begin{lemma}\label{lem:x_z}
    Given $i\in\{0, \infty\}$, let $\XX = \XX^i$ and $\VV = \VV^i$.
    Let $z_0\in\R^n$, $u\in\UU$, and $T\in\R_{>0}\cup\{+\infty\}$.
    Let $x\in C^0([0, T), \R^n)$ be such that $x(0) = z_0$.
    Then:
    \begin{itemize}
        \item
    if $x$ is a solution to $(\textnormal{\ref{syst}}, \XX, \UU)$ for some initial condition $x_0\in\XX$ and some $u\in\UU$, then $x$ is also a solution to $(\textnormal{\ref{syst_input}}, \R^n, \VV\times\UU)$ with inputs $u$ and $v = ((x_0\diamond x)(\cdot-\theta_k))_{k\in\{1,\dots,p\}}\in\VV$.
    \item conversely, if $(v,u)\in\VV\times\UU$
and $x$ is a corresponding solution to $(\textnormal{\ref{syst_input}}, \R^n, \VV\times\UU)$,
then for any $\delta\in(0, \min(T, \theta_1, \min_{\substack{k, j\in\{1, \dots, p\}\\ k\neq j}}|\theta_k-\theta_j|))$, 
and any $x_0\in\XX$ such that
$x_0|_{[-\theta_k, -\theta_k+\delta]}=v_k(\cdot+\theta_k)$ for all $k\in\{1,\dots,p\}$,
$x|_{[0, \delta)}$ is a solution to $(\textnormal{\ref{syst}}, \XX, \UU)$
with initial condition $x_0\in\XX$ and input $u$.
\end{itemize}
\end{lemma}

\begin{proof}
On the one hand,
assume that $x$ is a solution to \eqref{syst} and define $v = ((x_0\diamond x)(\cdot-\theta_k))_{k\in\{1,\dots,p\}}\in\VV$.
Then it holds for almost all $t\in[0, \theta_1)$ that
\begin{align*}
    \dot x(t)
    &= f\Big(x(t), \big((x_0\diamond x)(t-\theta_k)\big)_{k\in\{1,\dots,p\}}, u(t)\Big)
    \\
    &=f(x(t), v(t), u(t)),
\end{align*}
meaning that $x$ is a solution to \eqref{syst_input} with inputs $v$ and $u$.
On the other hand, assume that $v\in\VV$ satisfies the assumptions of the lemma.
The choice of $\delta$ is made to avoid conflicting definitions of $v$.
Then it holds for almost all $t\in[0, \delta)$ that
\begin{align*}
    \dot x(t)
    &= f(x(t), v(t), u(t))
    \\
    &=f\Big(x(t), \big((x_0\diamond x)(t-\theta_k)\big)_{k\in\{1,\dots,p\}}, u(t)\Big)
    \\
    &=f\Big(x(t), \big(x_0(t-\theta_k)\big)_{k\in\{1,\dots,p\}}, u(t)\Big).
\end{align*}
Hence, $x|_{[0, \delta)}$ is a solution to \eqref{syst} with initial condition $x_0$ and input $u$.
\end{proof}

\section{Main result}
\label{sec:main}

Our main result is the following relations between FC and BRS for systems \eqref{syst} and \eqref{syst_input}.

\begin{theorem}\label{th:main}
    All the following statements are equivalent:
    \begin{enumerate}[label = (\roman*)]
        \item \hypertarget{xinffc}{$\prop{\ref{syst}}{\XX^\infty}{\UU}{FC}$}
        \item \hypertarget{xinfbrs}{$\prop{\ref{syst}}{\XX^\infty}{\UU}{BRS}$}
        \item \hypertarget{x0brs}{$\prop{\ref{syst}}{\XX^0}{\UU}{BRS}$}
        \item \hypertarget{zinffc}{$\prop{\ref{syst_input}}{\R^n}{\VV^\infty\times\UU}{FC}$}
        \item \hypertarget{zinfbrs}{$\prop{\ref{syst_input}}{\R^n}{\VV^\infty\times\UU}{BRS}$}
        \item \hypertarget{z0brs}{$\prop{\ref{syst_input}}{\R^n}{\VV^0\times\UU}{BRS}$}
    \end{enumerate}
    Moreover, the following statements are equivalent, implied by the above, and do not imply the above:
    \begin{enumerate}[label = (\roman*), resume]
        \item \hypertarget{x0fc}{$\prop{\ref{syst}}{\XX^0}{\UU}{FC}$}
        \item \hypertarget{z0fc}{$\prop{\ref{syst_input}}{\R^n}{\VV^0\times\UU}{FC}$}
    \end{enumerate}
\end{theorem}

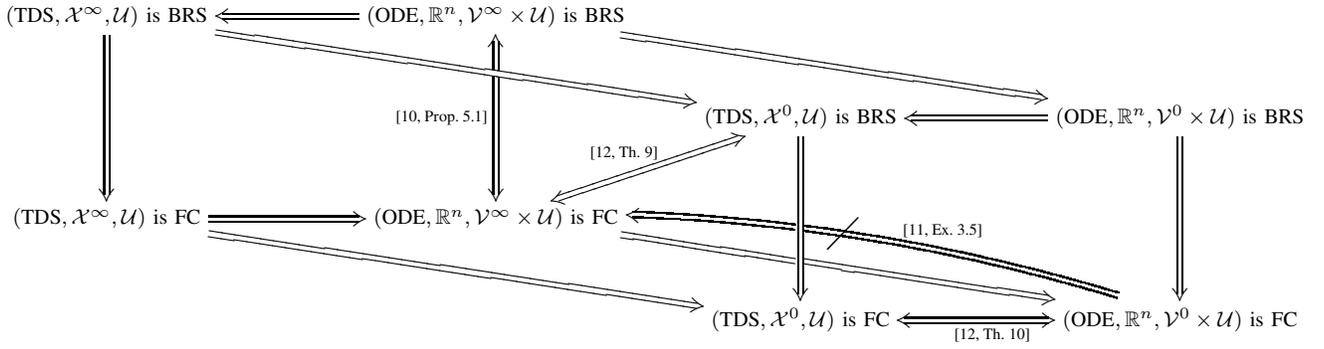
\begin{figure*}[!ht]
\centering
\footnotesize
\begin{equation*}
\xymatrix{
\prop{\ref{syst}}{\XX^\infty}{\UU}{BRS} \ar@{=>}[dd] \ar@{=>}[rrrd]
&&
\prop{\ref{syst_input}}{\R^n}{\VV^\infty\times\UU}{BRS}
\ar@{<=>}|(0.283){\text{\phantom{x}}}[dd]_{\hspace{-0.25em}\text{\cite[Prop. 5.1]{LSW96}}} \ar@{=>}[rrrd]
\ar@[blue]@{=>}[ll]
\\
&&& \prop{\ref{syst}}{\XX^0}{\UU}{BRS} \ar@{=>}[dd]
\ar@{<=>}_<<<<<<<<<<<<<{\text{\cite[Th. 9]{MaH23}}\hspace{1.3em}}[ld]
&&
\prop{\ref{syst_input}}{\R^n}{\VV^0\times\UU}{BRS} \ar@{=>}[dd]
\ar@[blue]@{=>}[ll]
\\
\prop{\ref{syst}}{\XX^\infty}{\UU}{FC} \ar@[blue]@{=>}[rr]
\ar@{=>}[rrrd]
&&
\prop{\ref{syst_input}}{\R^n}{\VV^\infty\times\UU}{FC} \ar@<0.3ex>@{=>}[rrrd]|(0.443){\text{\phantom{x}}}
\\
&&& \prop{\ref{syst}}{\XX^0}{\UU}{FC}\ar@{<=>}[rr]_{\text{\cite[Th. 10]{MaH23}}}
&&
\prop{\ref{syst_input}}{\R^n}{\VV^0\times\UU}{FC}
\ar@<-0.5ex>@{=>}@/_1pc/[lllu]_<<<<<<<<<<<<<{\hspace{1.5em}\text{\cite[Ex. 3.5]{MANCILLAAGUILAR20051111}}}
|(.50){\text{\Large $\diagup$}}
|(0.45){}
|(0.565){\text{\phantom{x}}}
}
\end{equation*}
\caption{Proof architecture of Theorem \ref{th:main}. Implications proved in the present article are indicated in blue. Implications known from the literature and trivial implications are indicated in black, as well as an implication known to be false ($\not\Rightarrow$). Together, these implications are sufficient to establish Theorem \ref{th:main}.
}
\label{fig}
\end{figure*}


The implications used to prove Theorem \ref{th:main} are depicted in Figure~\ref{fig}.
From the fact that $\XX^0\subset \XX^\infty$ and $\VV^0\subset \VV^\infty$, the following implications easily hold
\begin{center}
\hyperlink{xinffc}{$\xinffc$} $\Rightarrow$ \hyperlink{x0fc}{$\xzerofc$},    
\\
\hyperlink{xinfbrs}{$\xinfbrs$} $\Rightarrow$ \hyperlink{x0brs}{$\xzerobrs$},    
\\
\hyperlink{zinffc}{$\zinffc$} $\Rightarrow$ \hyperlink{z0fc}{$\zzerofc$}    
\\
\hyperlink{zinfbrs}{$\zinfbrs$} $\Rightarrow$ \hyperlink{z0brs}{$\zzerobrs$}
\end{center}


Since BRS implies FC, we also have that
\begin{center}
\hyperlink{xinfbrs}{$\xinfbrs$}
$\Rightarrow$ \hyperlink{xinffc}{$\xinffc$},
\\
\hyperlink{zinfbrs}{$\zinfbrs$}
$\Rightarrow$ \hyperlink{zinffc}{$\zinffc$},
\\
\hyperlink{x0brs}{$\xzerobrs$}
$\Rightarrow$ \hyperlink{x0fc}{$\xzerofc$},
\\
\hyperlink{z0brs}{$\zzerobrs$}
$\Rightarrow$ \hyperlink{z0fc}{$\zzerofc$}.
\end{center}

Moreover, from \cite[Theorem 9]{MaH23}, \cite[Theorem 10]{MaH23} and \cite[Proposition 5.1]{LSW96} respectively, we know that
\begin{center}  
\hyperlink{x0brs}{$\xzerobrs$} $\Leftrightarrow$ \hyperlink{zinffc}{$\zinffc$},
\\
\hyperlink{x0fc}{$\xzerofc$} $\Leftrightarrow$ \hyperlink{z0fc}{$\zzerofc$},
\\
\hyperlink{zinffc}{$\zinffc$} $\Leftrightarrow$ \hyperlink{zinfbrs}{$\zinfbrs$}.
\end{center}
According to \cite[Example 3.5]{MANCILLAAGUILAR20051111},
\begin{center}
\hyperlink{z0fc}{$\zzerofc$} $\nRightarrow$ \hyperlink{zinffc}{$\zinffc$}.
\end{center}
Combining all these previous results we see that, in order to prove Theorem \ref{th:main}, it remains to show that
\begin{itemize}
    \item \hyperlink{xinffc}{$\xinffc$} $\Rightarrow$
\hyperlink{zinffc}{$\zinffc$}
\item \hyperlink{zinfbrs}{$\zinfbrs$} $\Rightarrow$
\hyperlink{xinfbrs}{$\xinfbrs$}
\item \hyperlink{z0brs}{$\zzerobrs$} $\Rightarrow$
\hyperlink{x0brs}{$\xzerobrs$}.
\end{itemize}






\smallskip

\textit{Proof that}:
\vspace{-0.2cm}
\begin{center}
    \hyperlink{xinffc}{$\xinffc$} $\Rightarrow$
\hyperlink{zinffc}{$\zinffc$}
\\
\textit{and}
\\
\hyperlink{x0fc}{$\xzerofc$} $\Rightarrow$
\hyperlink{z0fc}{$\zzerofc$}\footnote{Although \hyperlink{x0fc}{$\xzerofc$} $\Rightarrow$
\hyperlink{z0fc}{$\zzerofc$} is already stated in \cite[Theorem 10]{MaH23}, we give a proof for the sake of completeness.}.
\end{center}
Given $i\in\{0, \infty\}$, let $\XX = \XX^i$ and $\VV = \VV^i$.
Assume that $\prop{\ref{syst}}{\XX}{\UU}{FC}$. It is known that FC corresponds to the absence of finite escape time (see e.g. \cite[Theorem 2]{CKP23}). Assume for the sake of contradiction that there exists $T\in(0, +\infty)$, $z_0\in\R^n$, $u\in\UU$ and $v\in\VV$ such that $|z(t; z_0, v, u)|\to+\infty$ as $t\to T^-$.
Set
$ \delta = \frac{1}{2}\min(T,\theta_1, \min_{\substack{k, j\in\{1, \dots, p\}\\ k\neq j}}|\theta_k-\theta_j|).   
$
Since all delays are assumed positive and distinct, it holds that $\delta>0$.
Let $t_0 = T - \delta>0$.
Note that for all $k\in\{1,\dots,p\}$,
we have that $0\notin[-\theta_k, -\theta_k+\delta]$ and  $[-\theta_k, -\theta_k+\delta]\cap [-\theta_j, -\theta_j+\delta] = \emptyset$ for all $j\in\{1,\dots,p\}$ with $k\neq j$. Then, we can choose $x_0\in\XX$ such that:
\begin{itemize}
\item $x_0(0) = z(t_0; z_0, v, u)$
\item for all $k\in\{1,\dots,p\}$ and almost all $t\in[-\theta_k, -\theta_k+\delta]$, $x_0(t) = v_k(t_0+t+\theta_k).$
\end{itemize}
Then, by assumption, the Cauchy problem \eqref{syst} initialized at $x_0$ and with input $u(t_0+\cdot)$ admits a unique solution $x(\cdot; x_0, u(t_0+\cdot))$ over $[-\thetamax,+\infty)$.
By Lemma \ref{lem:x_z},
$x(\cdot; x_0, u(t_0+\cdot))$
is also a solution to \eqref{syst_input} with input $((x_0\diamond x)(\cdot-\theta_k))_{k\in\{1,\dots,p\}}$ over $[-\thetamax,\delta]$.
In particular, for almost all $t\in[0, \theta_1]$ and all $k\in\{1,\dots,p\}$, we have
$(x_0\diamond x)(t-\theta_k) = x_0(t-\theta_k) = v(t_0+t)$.
Hence, $x(\cdot; x_0, u(t_0+\cdot))$ is a solution to \eqref{syst_input} with inputs $v(t_0+\cdot)$ and $u(t_0+\cdot)$ over $[0, \delta]$. By uniqueness of solutions to the Cauchy problem associated to \eqref{syst_input} initialized at $z(t_0; z_0, v, u)$, it holds that
$z(t_0+t; z_0, v, u) = x(t; x_0, u(t_0+\cdot))$
for all $t\in[0, \delta]$.
Since $x(\cdot; x_0, u(t_0+\cdot))$ is continuous over $[0, +\infty)$, we get that
$z(t; z_0, v, u)\to x(\delta; x_0, u(t_0+\cdot))$ as $t\to T^-$, which is a contradiction.

\smallskip

\textit{Proof that:}
\vspace{-0.2cm}
\begin{center}
\hyperlink{zinfbrs}{$\zinfbrs$} $\Rightarrow$ \hyperlink{xinfbrs}{$\xinfbrs$}
\\
\textit{and}
\\
\hyperlink{z0brs}{$\zzerobrs$} $\Rightarrow$ \hyperlink{x0brs}{$\xzerobrs$}.
\end{center}
Given $i\in\{0, \infty\}$, let $\XX = \XX^i$ and $\VV = \VV^i$.
Assume $\prop{\ref{syst_input}}{\R^n}{\VV\times\UU}{BRS}$.
We make use of Lemma \ref{lem:relax_brs} by considering $T=\theta_1/2$. Assume for the sake of contradiction that there exist $r>0$, sequences $(t_j)_{j\in\N}$ in $[0, T]$,
$(x_0^j)_{j\in\N}$ in $\XX$,
and
$(u_j)_{j\in\N}$ in $\UU$
such that 
$\|x_0^j\|_{\XX}\leq r$
and
$\|u_j|_{[0, T)}\|_{L^\infty}\leq r$
for all $j\in\N$
and that
$|x(t_j; x_0^j, u_j)|\to+\infty$ as $j\to+\infty$.
For each $j\in\N$,
define $v_j = ((x_0^j\diamond x(\cdot; x_0^j, u_j))(\cdot-\theta_k))_{k\in\{1,\dots,p\}}\in\VV$.
By Lemma \ref{lem:x_z}, $z(\cdot; x_0^j(0), v_j, u_j) = x(\cdot; x_0^j, u_j)$.
Since $T<\theta_1$, we have
$\|v_j|_{[0, T)}\|_{L^\infty}
\leq \|x_0^j\|_\XX\leq r$.
By assumption, $(z(t_j; x_0^j(0), v_j, u_j))_{j\in\N}$ is bounded.
Hence, $(x(t_j; x_0^j, u_j))_{j\in\N}$ is bounded, which is a contradiction.

\section{Relations between GAS and UGAS}
\label{sec:gas}


In this section, we specialize \eqref{syst} to systems without inputs, that is,
\begin{equation}\label{syst_noinputs}
\tag{TDS}
\dot x(t) = f(x(t), (x(t-\theta_k))_{k\in\{1,\dots,p\}}).
\end{equation}

Using our main result, we show that global asymptotic stability is necessarily uniform in the $\XX^\infty$ setting (in contrast to the $\XX^0$-case).
To formalize this, recall the following notions.


\begin{definition}
Let $\XX$ be either $\XX^0$ or $\XX^\infty$.
We say that
$(\textnormal{\ref{syst_noinputs}}, \XX)$
is:
\begin{itemize}
    \item \emph{locally stable (LS)} if
    for all $\eps>0$, there exists $\delta>0$ such that
    \[\sup\{\|x_t(x_0)\|_\XX\mid t\geq0, \|x_0\|_\XX\leq \delta\}\leq \eps;
    \]
    
    \item \emph{globally attractive (GA)}
    if
    \begin{center}
    $\|x_t(x_0)\|_\XX\to0$ as $t\to+\infty$ for all $x_0\in\XX$;        
    \end{center}

    \item \emph{globally asymptotically stable (GAS)}
    if
    it is both \emph{LS} and \emph{GA};
    
    \item \emph{uniformly globally asymptotically stable (UGAS)}
    if
    there exists $\beta\in\KL$ such that
    \[
    \|x_t(x_0)\|_\XX\leq\beta(\|x_0\|_\XX, t),\qquad  \forall x_0\in\XX, \ \forall t\in\R_+.
    \]
\end{itemize}
\end{definition}

Unlike GAS, the UGAS property guarantees that starting from a bounded set of initial states, the convergence rate to zero cannot be arbitrarily slow, and the transient overshoot cannot be arbitrarily large.
From \cite{CWM24} we know that, unlike in ODE case \cite{LSW96}, for $(\text{TDS}, \XX^0)$, GAS does not imply uniform global attractivity, and by \cite[Proposition 7]{MaH23}, for $(\text{TDS}, \XX^0)$, GAS and uniform global attractivity still do not apply UGAS.
Nevertheless, the use of other state spaces, such as H\"older or Sobolev spaces, allows recovering the equivalence between GAS and UGAS \cite{karafyllis2022global}. It was shown there that GAS and UGAS are equivalent properties, provided that the system has BRS. Using our main result, we show that uniformity always holds in the $\XX^\infty$ setting, at least for systems with a finite number of discrete delays. 

\begin{theorem}\label{th:gas}
The following assertions are equivalent:
\begin{enumerate}[label = (\roman*)]
    \item \hypertarget{gas}{$\gas$}
    \item \hypertarget{ugas}{$\ugas$}
    \item \hypertarget{0_ugas}{$\zerougas$}
\end{enumerate}
Moreover, the following statement is implied by the above and does not imply the above:
\begin{enumerate}[label = (\roman*), resume]
    \item \hypertarget{0_gas}{$\zerogas$}
\end{enumerate}
\end{theorem}

The proof architecture is depticted by Figure~\ref{fig2}.

\begin{figure}[!ht]
\centering
\begin{equation*}
\xymatrix{
\ugas
\ar@[blue]@{<=>}[rr]
\ar@{=>}[dd]
&&
\zerougas
\ar@{=>}[dd]
\\
\\
\gas
\ar@{=>}[rr]
\ar@{=>}@[blue][rruu]
&&
\zerogas
\ar@<-3ex>@{=>}[uu]_<<<<<<<<{\text{\cite{MaH23}}}|(.50){\text{\Large $\diagup$}}|(.40){}
}
\end{equation*}
\caption{Proof architecture of Theorem \ref{th:gas}. Implications proved in the present article are depicted in blue. Implications known from the literature, trivial implications, as well as an implication known to be false ($\not\Rightarrow$) are indicated in black. Together, these implications are sufficient to establish Theorem \ref{th:gas}.}
\label{fig2}
\end{figure}
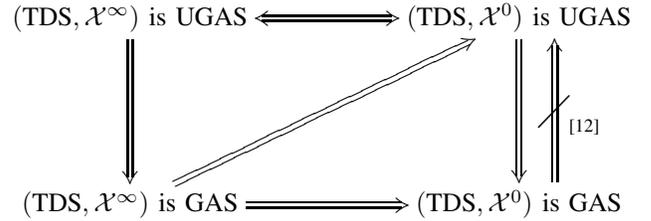

\begin{proof}
It is clear from the definitions that
\begin{center}
\hyperlink{ugas}{$\ugas$} $\Rightarrow$ \hyperlink{gas}{$\gas$},
\\
\hyperlink{0_ugas}{$\zerougas$} $\Rightarrow$ \hyperlink{0_gas}{$\zerogas$}.
\end{center}
From the inclusion $\XX^0\subset \XX^\infty$, and since on $\XX^0$ the norms of $\XX^0$ and $\XX^\infty$ coincide, it follows that 
\begin{center}
\hyperlink{ugas}{$\ugas$} $\Rightarrow$ \hyperlink{0_ugas}{$\zerougas$},
\\
\hyperlink{gas}{$\gas$} $\Rightarrow$ \hyperlink{0_gas}{$\zerogas$}.
\end{center}
According to \cite[Proposition 7]{MaH23},
\begin{center}
    \hyperlink{0_gas}{$\zerogas$} $\nRightarrow$ \hyperlink{0_ugas}{$\zerougas$}.
\end{center}
Hence, it remains to prove  that
\begin{center}
 \hyperlink{gas}{$\gas$} $\Rightarrow$ \hyperlink{0_ugas}{$\zerougas$}
 \\
 $\Rightarrow$ \hyperlink{ugas}{$\ugas$}.
\end{center}

\textit{Proof that}
\hyperlink{gas}{$\gas$} $\Rightarrow$ \hyperlink{0_ugas}{$\zerougas$}.

Assume that $\gas$. In particular, $\xinffcnou$. Hence, according to Theorem \ref{th:main}, $\xzerobrsnou$. Moreover, $\zerogas$. Hence, according to \cite[Theorem 1]{karafyllis2022global}, $\zerougas$.

\smallskip

\textit{Proof that}
\hyperlink{0_ugas}{$\zerougas$} $\Rightarrow$ \hyperlink{ugas}{$\ugas$}.
As $\zerougas$, $\xzerobrsnou$. 
By Theorem \ref{th:main}, $\xinfbrsnou$.
Applying\footnote{
Note that system \eqref{syst_noinputs} does not exactly fit the framework of \cite[Lemma 3]{MiW18b}, since its flow is not necessarily continuous when $\XX$ is $\XX^\infty$ (see Remark \ref{rem:cont}). However, the proof of \cite[Lemma 3]{MiW18b} does not exploit at all continuity of the flow, hence remains valid in our context.
}
\cite[Lemma 3]{MiW18b} to \eqref{syst_noinputs}, we get that 
there exists a continuous function $\mu:\R_+^2\to\R_+$, non-decreasing with respect to each variable, such that
\[
\|x_t(x_0)\|_{\XX^\infty}\leq \mu(t, \|x_0\|_{\XX^\infty}),\qquad x_0\in\XX^\infty, \quad t\geq 0.
\]
In particular, for all $t\in[0, \theta_p]$,
$$
\|x_t(x_0)\|_{\XX^\infty}\leq \mu(\theta_p, \|x_0\|_{\XX^\infty}).
$$




Let $\kappa:\R_+\to\R_+$ be a continuous function such that for all $r>0$, $\kappa(r)$ is greater than the Lipschitz constant of $f$ over the ball of $\XX^\infty$ of center $0$ and radius $r$.
Then, by Grönwall's inequality and using that $f(0,0)=0$, for all $t\in[0, \theta_p]$ and all $x_0\in\XX^\infty$,
\begin{align}\label{eq:beta1}
\|x_t(x_0)\|_{\XX^\infty} \leq e^{\theta_p\kappa(\mu(\theta_p, \|x_0\|_{\XX^\infty}))} \|x_0\|_{\XX^\infty}.
\end{align}
On the other hand, using the cocycle property,
$x_t(x_0) = x_{t-\theta_p}(x_{\theta_p}(x_0))$ for all $t\geq\theta_p$ and all $x_0\in\XX$.
Since $x(\cdot; x_0)$ is continuous over $[0, +\infty)$,
$x_{\theta_p}(x_0)\in C^0([-\theta_p, 0], \R^n)$.
Thus,
$x(\cdot; x_{\theta_p}(x_0), 0)$ is an $\XX^0$-solution of \eqref{syst_noinputs}.
Hence, there exists $\beta\in\KL$ such that
for all $x_0\in\XX^\infty$ and all $t\geq \theta_p$,
\begin{align}
\|x_t(x_0)\|_{\XX^\infty}
&\leq \beta(\|x_{\theta_p}(x_0)\|_{\XX^\infty}, t-\theta_p)\nonumber
\\
&\leq \beta(e^{\theta_p\kappa(\mu(\theta_p, \|x_0\|_{\XX^\infty}))} \|x_0\|_{\XX^\infty}, t-\theta_p).
\label{eq:beta2}
\end{align}
For all $r, t\geq0$, define
\begin{equation*}
\tilde\beta(r, t) = \begin{cases}
    \beta(r, 0)e^{-t} &\text{if } t\in[-\theta_p, 0)
    \\
    \beta(r, t) &\text{if } t\geq 0
\end{cases}
\end{equation*}
and
\begin{equation*}
\bar\beta(r, t) = \max(e^{\theta_p-t}e^{\theta_p\kappa(\mu(\theta_p, r))} r, \tilde\beta(e^{\theta_p\kappa(\mu(\theta_p, r))} r, t-\theta_p)).
\end{equation*}
Then, $\tilde\beta(\cdot, \cdot-\theta_p)\in\KL$ since $\beta\in\KL$, $(r, t)\mapsto e^{\theta_p-t}e^{\theta_p\kappa(\mu(\theta_p, r))} r$ is of class $\KL$, and $\bar\beta\in\KL$ as the maximum between two functions of class $\KL$.
Combining \eqref{eq:beta1} and \eqref{eq:beta2}, we obtain that for all $x_0\in\XX^\infty$ and all $t\in\R_+$,
\begin{align*}
\|x_t(x_0)\|_{\XX^\infty}
\leq
\bar\beta(\|x_0\|_{\XX^\infty}, t).
\end{align*}
Thus, $\ugas$.
\end{proof}

\section{Conclusion}

We have shown that, when working with the state space $L^\infty\times\R^n$, FC implies BRS for time-delay systems with a finite number of discrete delays.
It is worth noticing that this space is larger than the usual space of continuous functions.
Hence, it is now established that the equivalence between FC and BRS for time-delay systems fails in the state space of continuous functions \cite{MaH23}, but holds both for a smaller state space (Sobolev or Hölder, \cite{karafyllis2022global}) and for a larger one ($L^\infty\times\R^n$, present paper).

This equivalence in $L^\infty\times\R^n$ opens the door to the relaxation of many characterizations of stability properties such as ISS, UGAS, etc. (see \cite{MWC24} for corresponding results in the state space of continuous functions), at the cost of recasting all the results that are known in the continuous state space \cite{CKP23} into $L^\infty\times\R^n$. We have illustrated this by proving that GAS implies UGAS in the new framework.

The question of extending these results to systems with distributed delays remains open. Firstly, to the best of the authors' knowledge, the proof of the existence of solutions in the state space $L^p\times\R^n$ fails when $p=\infty$ (see \cite{DELFOUR1972213}). Secondly, even if a solution exists for a specific system, the approach developed in this paper (based on relating the delay system with an ODE) fails. Hence, new tools have to be developed.

\bibliographystyle{abbrv}
\bibliography{references.bib}

\end{document}